\documentclass[fleqn]{amsart}

\usepackage{amsmath}

\usepackage{graphicx}
\usepackage{subfigure}

\numberwithin{equation}{section}
\numberwithin{table}{section}
\numberwithin{figure}{section}

\theoremstyle{plain}
\newtheorem{definition}{Definition}[section]
\newtheorem{proposition}[definition]{Proposition}
\newtheorem{theorem}[definition]{Theorem}
\newtheorem{corollary}[definition]{Corollary}
\newtheorem{lemma}[definition]{Lemma}

\theoremstyle{definition}
\newtheorem{remark}[definition]{Remark}
\newtheorem{example}[definition]{Example}
\newtheorem{problem}[definition]{Problem}

\newcommand{\sub}[1]{\underline{#1}}
\newcommand{\id}{\mathrm{id}}
\newcommand{\sgn}{\mathrm{sgn}}
\newcommand{\C}{\mathbb{C}}
\newcommand{\I}{{\mathcal{I}}}
\newcommand{\Iti}{\I(W)}
\newcommand{\al}{\ell^\prime}
\newcommand{\exc}{\mathrm{exc}}
\newcommand{\inv}{\mathrm{inv}}
\newcommand{\Sn}{\mathfrak{S}_n}
\newcommand{\SBn}{\mathfrak{S}^B_n}
\newcommand{\SDn}{\mathfrak{S}^D_n}
\newcommand{\Sm}{\mathfrak{S}_m}
\newcommand{\SBm}{\mathfrak{S}^B_m}
\newcommand{\wh}{\widehat}
\newcommand{\BrI}{\mathrm{Br}(\Iti)}
\newcommand{\rest}{\mathrm{r}}
\newcommand{\sg}{\mathfrak{s}}
\newcommand{\ISn}{\I(\Sn)}
\newcommand{\ISBn}{\I(\SBn)}

\title{Pattern avoidance and the Bruhat order on involutions}

\thanks{This article is largely based on results from the
second author's M.Sc.\ thesis \cite{Vorwerk2007}.}

\author{Axel Hultman and Kathrin Vorwerk}

\address{Department of Mathematics, KTH, SE-100 44 Stockholm, Sweden}
\email{axel@math.kth.se}

\address{Department of Mathematics, KTH, SE-100 44 Stockholm, Sweden}
\email{kathrinv@math.kth.se}

\begin{document}

\begin{abstract}
We show that the principal order ideal below an element $w$ in the
Bruhat order on involutions in a symmetric group is a Boolean lattice
if and only if $w$ avoids the patterns $4321$, $45312$ and
$456123$. Similar criteria for signed permutations are also
stated. Involutions with this property are enumerated with respect to
natural statistics. In this context, a bijective correspondence with
certain Motzkin paths is demonstrated.
\end{abstract}

\maketitle

\section{Introduction}
The Bruhat order on a Coxeter group is fundamental in a multitude of
contexts. For example, the incidences among the closed cells
in the Bruhat decomposition of a flag variety are governed by
the Bruhat order on the corresponding Weyl group.

In spite of its importance, the Bruhat order is in many ways poorly
understood. For example, much about the structure of
intervals, or even principal order ideals, remains unclear. There
are, however, several 
known connections between structural properties of principal  
order ideals in the Bruhat order and pattern avoidance properties of
the corresponding group elements. Here are some examples:
\begin{itemize}
\item A Schubert variety is rationally smooth if and only if the
corresponding Bruhat order ideal is rank-symmetric. These properties have
been characterized in terms of pattern avoidance by Lakshmibai and Sandhya
\cite{Lakshmibai1990} (type $A$) and Billey \cite{Billey1998} (types $B$,
$C$, $D$).
\item Gasharov and Reiner \cite{GasharovReiner2002} have shown that
a Schubert variety is ``defined by inclusions'' precisely when the
corresponding permutation avoids certain patterns. By work of
Sj\"ostrand \cite{Sjostrand2007}, these permutations are precisely
those whose Bruhat order ideal is defined by the ``right hull'' of the
permutation.    
\item Tenner \cite{Tenner2007} has demonstrated
that the permutations whose Bruhat order ideals 
are Boolean lattices can be characterized in terms of pattern
avoidance. By general theory, this characterizes the lattices among
all principal order ideals in the Bruhat order.
\end{itemize}

An interesting subposet of the Bruhat order is induced by the
involutions. Activity around this subposet was spawned by
Richardson and Springer \cite{Richardson1990} who established connections
with algebraic geometry that resemble (and, in some sense,
generalize) the situation in the full Bruhat order. For example, the
(dual of the) Bruhat order on the involutions in the symmetric group
$\mathfrak{S}_{2n+1}$ encodes the incidences among the closed orbits
under the action of a Borel subgroup on the symmetric variety
$SL_{2n+1}(\C)/SO_{2n+1}(\C)$; cf.\ \cite[Example 10.3]{Richardson1990}.

Recently, it has been shown that the Bruhat order on involutions
has many combinatorial and topological properties in common with the full
Bruhat order \cite{Hultman2007, Incitti2006}. The purpose of this
paper is to incorporate pattern avoidance in this
picture. Specifically, we shall study analogues for involutions of
the aforementioned results of Tenner.

Our main result is as follows:
\begin{theorem}
\label{th:BooleanInvolution}
The principal order ideal generated by an involution $w$ in the Bruhat
order on the involutions in a symmetric group is a Boolean lattice if
and only if $w$ avoids the patterns $4321$, $45312$ and $456123$. 
\end{theorem}

The remainder of this paper is organised in the following way. In the
next section, we recall standard definitions and agree on notation. That
section also includes a brief review of some probably not so standard
results on involutions in Coxeter groups. After that, we turn to the
proof of Theorem \ref{th:BooleanInvolution} in Section
\ref{SectionBooleanInvolutions}. A corresponding result for signed
permutations (the type $B$ case) is also given. Section
\ref{SectionEnumeration} is devoted to enumerative results; we count
involutions with Boolean principal order ideals with respect to
various natural statistics. A bijective correspondence with certain
Motzkin paths is constructed. Finally, we suggest a direction for
further research in Section \ref{SectionOpenProblems}. 

\section{Preliminaries}

\subsection{Permutations and patterns}
Let $\Sn$ denote the symmetric group consisting of all permutations of
$[n]=\{1, \dots, n\}$. 

An {\em inversion} of $\pi\in \Sn$ is a pair $(i,j)$ such that $i<j$
and $\pi(i)>\pi(j)$. The number of inversions of $\pi$ is denoted by
$\inv(\pi)$. 

The {\em excedances} and the {\em deficiencies} of $\pi\in \Sn$ are
the indices $i\in [n]$ such that $\pi(i)>i$ and $\pi(i)<i$,
respectively. We use $\exc(\pi)$ to denote the number of excedances of
$\pi$.   

Given $\pi \in \Sn$ and $p\in \Sm$ (with $m\le n$), say that $\pi$
{\em contains the pattern} $p$ if there exist $1\le i_1<\cdots<i_m\le
n$ such that for all $j,k\in [m]$, $\pi(i_j)<\pi(i_k)$ if and only if
$p(j) < p(k)$. In this case, say that $\langle p\rangle = (\pi(i_1),
\dots, \pi(i_m))$ is an {\em occurrence} of $p$ in $\pi$. Furthermore, we
write $\langle p(j) \rangle = \pi(i_j)$ for $j\in [m]$.

If $\pi$ does not contain $p$, it {\em avoids} $p$. 

Suppose $\pi\in \Sn$, $p \in \Sm$ and that $\langle p \rangle$ is an
occurrence of $p$ in $\pi$. We say that this
occurrence is {\em induced} if $\langle p(j)\rangle = \pi(\langle
j\rangle)$ for all $j \in [m]$. 

\begin{example}
Consider $\pi = 84725631\in \mathfrak{S}_8$. It has
several occurrences of the pattern $4231$; two of them are
$(8,5,6,1)$ and $(8,4,5,3)$. The former occurrence is induced
while the latter is not.
\end{example}

Recall that an {\em involution} is an element of order at most
two. At times, we shall find it 
convenient to represent an involution $w\in \Sn$ by the graph on
vertex set $[n]$ in which two vertices are joined by an edge if they
belong to the same $2$-cycle in $w$. For an example, see Figure
\ref{fig:Sn_Patterns}.  
 
\subsection{Coxeter groups}
Here, we briefly review those facts from Coxeter group theory that we
need in the sequel. For more details, see \cite{Bjoerner2005} or
\cite{Humphreys1990}. 

A {\em Coxeter group} is a group $W$ generated by a finite set $S$ of
involutions where all relations among the generators are derived from
equations of the form $(ss^\prime)^{m(s,s^\prime)}=e$ for some
$m(s,s^\prime)=m(s^\prime,s)\ge 2$, where $s,s^\prime\in S$ are
disctinct generators. Here, $e\in W$ denotes the identity element. The
pair $(W,S)$ is referred to as a {\em Coxeter system}.

We may specify a Coxeter system using its {\em Coxeter graph}. This is
an edge-labelled complete graph on vertex set $S$ where the edge
$\{s,s^\prime\}$ has the label $m(s,s^\prime)$. For convenience, edges
labelled $2$ and edge labels that equal $3$ are suppressed from the
notation.

Let $(W,S)$ be a Coxeter system. Given $w\in W$, suppose $k$ is the
smallest number such that $w=s_1\cdots s_k$ for some $s_i\in S$. Then
$k$ is the {\em length} of $w$, denoted $\ell(w)$, and the word
$s_1\cdots s_k$ is called a {\em reduced expression} for $w$.

The set of {\em reflections} of $W$ is $T=\{wsw^{-1}: w \in
W,s \in S \}$. Define the {\em absolute length} $\al(w)$ to be 
the smallest $k$ such that $w$ is a product of $k$ reflections.

\begin{example}
The symmetric group $\Sn$ is a Coxeter group with the adjacent
transpositions $\sg_i = (i,i+1)$, $i\in [n-1]$, as Coxeter generators. Its
Coxeter graph is simply a path on $n-1$ vertices. In this setting,
$\ell(w) = \inv(w)$.

In the $\Sn$ case, $T$ is the set of transpositions. It is well-known that
the minimum number of transpositions required to express $w\in \Sn$ as
a product is $n-c(w)$, where $c(w)$ is the number of cycles in the
disjoint cycle decomposition of $w$. In particular, if $w\in \Sn$ is
an involution, $\al(w)$ is the number of $2$-cycles in $w$. In other
words, $\al(w)=\exc(w)$.  
\end{example}

The {\em Bruhat order} is the partial order on $W$ defined by $u\leq w$
if and only if $w = u t_1\cdots t_m$ for some $t_i\in T$ such that
$\ell(u t_1\cdots t_i)<\ell(ut_1\cdots t_{i+1})$ for all $i\in
[m-1]$. Clearly, $e\in W$ is the minimum element under the Bruhat
order.

\subsection{Involutions in Coxeter groups}\label{ss:involutions}
As before, let $(W,S)$ be a Coxeter system. Denote by $\Iti \subseteq
W$ the set of involutions in $W$. We now review some results on the
combinatorics of $\Iti$. They can all be found in \cite{Hultman2007}
or \cite{Hultman2007a}. The reader who is acquainted with the subject
will notice that all these properties are completely analogous to
standard statements about the full group $W$.  

Introduce a set of symbols $\sub{S} = \{\sub{s}: s\in S\}$. Define an
action of the free monoid $\sub{S}^*$ from the right on (the set) $W$
by 
\[
w\sub{s} = 
\begin{cases}
ws & \text{if $sws =  w$,}\\
sws & \text{otherwise,}
\end{cases}
\] 
and $w\sub{s}_1\cdots \sub{s}_k =
(\cdots(w\sub{s}_1)\sub{s}_2\cdots)\sub{s}_k$ for $w\in W$, 
$\sub{s}_i\in \sub{S}$. By abuse of notation, we write $\sub{s}_1\cdots
\sub{s}_k$ instead of $e\sub{s}_1\cdots \sub{s}_k$. The elements of
this kind are precisely the involutions in $W$:

\begin{proposition} \label{pr:orbit}
The orbit of $e$ under the $\sub{S}^*$-action is $\Iti$.
\qed
\end{proposition}

When $w\in \Iti$, the condition $sws = w$ which appears in the
definition of the $\sub{S}^*$-action is equivalent to $\ell(sws)=\ell(w)$. 

If $w = \sub{s}_1\cdots \sub{s}_k$ for some $\sub{s}_i \in
\sub{S}$, then the sequence  $\sub{s}_1\cdots \sub{s}_k$ is called an
{\em $\sub{S}$-expression} for $w$. This expression is {\em reduced} if
$k$ is minimal among all such expressions. In this case, $k$ is called
the {\em rank} and denoted $\rho(w)$. 

\begin{proposition}[{\bf Deletion property}] Suppose $\sub{s}_1\cdots
\sub{s}_k$ is an $\sub{S}$-expression for $w$ which is not reduced. Then,
$w = \sub{s}_1\cdots \wh{\sub{s}_i} \cdots \wh{\sub{s}_j} \cdots
\sub{s}_k$ for some $1\le i < j \le k$, where a hat means omission of
that element.
\qed
\end{proposition}

Let $\BrI$ denote the subposet of the Bruhat order on $W$ induced by
$\Iti$. Next, we recall a convenient characterization of its order relation.

\begin{proposition}[{\bf Subword property}]
Suppose that $\sub{s}_1\cdots \sub{s}_k$ is a reduced $\sub{S}$-expression for
$w\in \Iti$. For $u \in \Iti$, we have $u \leq w$ if and only if $u =
\sub{s}_{i_1}\cdots \sub{s}_{i_m}$ for some $1\le i_1<\cdots<i_m\le k$.
\qed
\end{proposition}

The poset $\BrI$ is graded with rank function $\rho$. Furthermore,
$\rho(w) = (\ell(w)+\al(w))/2$ for all $w\in \Iti$. In fact, given a
reduced $\sub{S}$-expression $\sub{s}_1\cdots \sub{s}_k$ for $w\in
\Iti$, one has 
\[
\al(w) = |\{i\in [k]: \sub{s}_1\cdots \sub{s}_i = \sub{s}_1\cdots
\sub{s}_{i-1}s_i\}| 
\]
and, consequently,
\[
\ell(w) = \al(w) + 2\cdot |\{i\in [k]: \sub{s}_1\cdots \sub{s}_i
\neq \sub{s}_1\cdots \sub{s}_{i-1}s_i\}|. 
\]

\section{Boolean involutions and pattern avoidance}
\label{SectionBooleanInvolutions}
As before, let $(W,S)$ be a Coxeter system. For $w\in \Iti$, denote by
$B(w)$ the principal order ideal below $w$ in the Bruhat order on
involutions. In other words, $B(w)$ is the subposet of $\BrI$
induced by $\{u\in \Iti: u\le w\}$.

We call an involution $w\in \Iti$ {\em
Boolean} if $B(w)$ is isomorphic to a Boolean lattice. In this section
we shall prove the characterization of Boolean involutions in
$\I(\Sn)$ which was stated as Theorem \ref{th:BooleanInvolution}.

First, we observe a useful characterization of Boolean involutions which
is valid in any Coxeter group. 

\begin{proposition}
\label{BooleanInvolutionWords}
Let $w \in \Iti$. Then $w$ is Boolean if and only if no reduced
$\sub{S}$-expression for $w$ has repeated letters. This is the case if
and only if there is an $\sub{S}$-expression for $w$ without repeated
letters. 
\begin{proof}
Observe that, by the subword property, every reduced
$\sub{S}$-expression of $w\in 
\Iti$ contains the same set of letters, namely $\{\sub{s}\in \sub{S} :
\sub{s} \le w\}$. If $\sub{s}_1\cdots \sub{s}_{k-1}$ is a reduced
$\sub{S}$-expression for $w\in \Iti$ and
all $\sub{s}_i$, $i\in [k]$, are distinct, then $\sub{s}_1\cdots
\sub{s}_k$ is reduced, too; otherwise the deletion property would
imply that $w = \sub{s}_1\cdots \sub{s}_k\sub{s}_k$ has a reduced
expression containing the letter $\sub{s}_k$, contradicting the above
assertion. We conclude that every $\sub{S}$-expression containing only
distinct letters is reduced. The ``if'' direction (of both assertions)
therefore follows directly from the subword property.

Since $\rho$ is the rank function of $\BrI$, the elements of rank one
in $[e,w]$ are the $\sub{s}_i \leq w$. Thus, if $w$ has a reduced
$\sub{S}$-expression containing repeated letters, $[e,w]$ will have
fewer elements of rank one than the Boolean lattice of rank
$\rho(w)$, so that $w$ cannot be Boolean. This shows the ``only if''
part of the assertions.
\end{proof}
\end{proposition}

\begin{remark}
As a consequence of \cite[Theorem 4.5]{Hultman2007}, the principal
order ideals in $\BrI$ are 
compressible Eulerian posets in the sense of du Cloux
\cite{Cloux2000}. It then follows from \cite[Corollary
5.4.1]{Cloux2000}, that such an ideal is a lattice if and only if it
is a Boolean lattice. Thus, the Boolean involutions are precisely the
involutions whose principal order ideals are lattices.
\end{remark}

\begin{remark}
The map $w\mapsto w^{-1}$ is an automorphism of the Bruhat order on
the full group $W$. The fixed point poset is $\BrI$. It is easy
to see that the fixed point poset of any automorphism of a Boolean
lattice is itself a Boolean lattice. Therefore, an involution
$w$ is Boolean if its principal order ideal in the full Bruhat order
on $W$ is Boolean. The converse, however, does not hold. 
\end{remark}

\subsection{Proof of Theorem \ref{th:BooleanInvolution}}
\label{BooleanInvolutionsOfIn}

We now proceed to prove Theorem \ref{th:BooleanInvolution}. First,
however, let us give a short outline of the idea of the proof. We shall
introduce the notions of connected components and long-crossing pairs
for purely technical purposes. Then, Propositions
\ref{prop:BooleanInvolutionSufficient} and
\ref{prop:BooleanInvolutionNecessary} establish the fact that
being Boolean is equivalent to the non-existence of a long-crossing
pair. Finally, we show in Proposition
\ref{prop:BooleanInvolutionPatterns} that $w \in \ISn$ has a
long-crossing pair if and only if it contains one or more of the
patterns $4321$, $45312$ and $456123$. 
 
\begin{definition}
\label{def:Connected}
Let $w \in \ISn$. The positions $i,j \in [n]$ are called {\em
connected} if there exists a sequence $i = i_0, i_1, \ldots, i_k = j$
such that $\sgn(i_{l-1} - i_l) = - \sgn( w(i_{l-1}) - w(i_l) )$ for
all $l\in [k]$.
\end{definition}

This notion of connectedness induces an equivalence relation on $[n]$. We call the equivalence classes with respect to this relation {\em connected components} of $w$ and denote the set of connected components of $w$ by
${\mathcal{C}}(w)$. An involution $w \in \ISn$ is called {\em
connected} if $[n]$ is the unique connected component of $w$. 

\begin{lemma}
Let $w \in \ISn$. The connected components $C \in {\mathcal{C}}(w)$ of $w$ are intervals.
\begin{proof}
Let $i<j<k$ be such that $i$ and $k$ are connected. Using Definition
\ref{def:Connected} it follows that there are $p,q \in [n]$ such that
$p<j<q$ and $w(p) > w(q)$. This implies either $w(j) < w(p)$ or $w(j)
> w(q)$. Thus, $j$ is in the same connected component as $i$ and
$k$. 
\end{proof}
\end{lemma}

For $w \in \Sn$ and $D \subseteq [n]$ we define the {\em restriction}
$w_D$ of $w$ to $D$ by 
\[
	w_D(i) =
		\begin{cases}
		w(i) & \text{if } i \in D, \\
		i & \text{otherwise}.
		\end{cases}
\]
If $w \in \ISn$ is an involution and $C$ is the union of connected
components of $w$, then $w_C$ is also an involution. 

Recall that, as a Coxeter group, $\Sn$ is generated by the adjacent
transpositions $\sg_i = (i,i+1)$, $i \in \Sn$.

Let $w \in \ISn$ and ${\mathcal{C}}(w) = \{ C_1, \ldots, C_k \}$. Then
$w_{C_i}$ belongs to the {\em standard parabolic subgroup} of $\Sn$
generated by $\sg_{a_i}, \sg_{a_i+1},\ldots, \sg_{b_i}$ where $C_i =
[a_i,b_i+1]$. In particular, those subgroups have pairwise trivial intersections and generators of different subgroups commute. This implies that the concatenation of reduced $\sub{S}$-expressions for $w_{C_i}$ and $w_{C_j}$ is a reduced $\sub{S}$-expression for $w_{C_i \cup C_j}$ for all $i,j \in [k]$ with $i \not=j$. The following lemma is now immediate.

\begin{lemma}
\label{prop:ConnectedComponents}
Let $w \in \ISn$ with ${\mathcal{C}}(w) = \{ C_1, \ldots, C_k
\}$. Then the following holds: 
\begin{enumerate}
\item If $w_i$ is a reduced $\sub{S}$-expression for $w_{C_i}$ for all
$i \in [k]$, then the concatenation $w_{\pi(1)} w_{\pi(2)} \ldots
w_{\pi(k)}$ is a reduced $\sub{S}$-expression for $w$ for any $\pi \in
S_k$. 
\item $B(w) \cong B(w_{C_1}) \times \ldots \times B(w_{C_k})$.
\item $w$ is Boolean if and only if $w_{C_i}$ is Boolean for all $i \in [k]$.
\end{enumerate}
\qed
\end{lemma}

\begin{definition}
\label{def:LongCrossing}
Let $w \in \ISn$ and $i,j \in [n]$. The pair $(i,j)$ is {\em long-crossing} in $w$ if $i < j < w(j)$ and $w(i) > j+1$.
\end{definition}

We note that the elements $i$ and $j$ of a long-crossing pair $(i,j)$
in some $w \in \ISn$ are connected.
  
\begin{proposition}[\bf A sufficiency criterion]
\label{prop:BooleanInvolutionSufficient}
Let $w \in \ISn$. If there is no long-crossing pair $(i,j)$ in $w$,
then $w$ is Boolean.
\end{proposition}

\begin{proof}
Suppose $n\ge 3$. Using Lemma \ref{prop:ConnectedComponents} we can
assume that $w$ is connected (otherwise consider each connected
component separately). Assume that $\{ (i_l,w(i_l)) : l \in [k]\}$ is
the set of $2$-cycles of $w$ with $i_l < w(i_l)$ for all $l \in 
[k]$ and $i_1 < i_2 < \ldots < i_k$. Connectedness of $w$ implies $i_1
= 1$, $w(i_k) = n$ and $w(i_l) > i_{l+1}$, $l\in [k-1]$. Furthermore,
$(i_l,i_{l+1})$ not being a long-crossing pair implies $w(i_l) \leq
i_{l+1} + 1$ and thus $w(i_l) = i_{l+1} + 1$ for all $l \in [k-1]$.

Consider the involution $v = (i_1,i_2)(i_2+1,i_3)(i_3+1,i_4)\ldots(i_{k-1}+1,i_k)(i_k+1,n)\in \ISn$. An $\sub{S}$-expression for $v$ is given by
\[
 \sub{\sg_1} \sub{\sg_2} \ldots \sub{\sg_{i_2-1}} \sub{\sg_{i_2+1}}
\ldots \sub{\sg_{i_3-1}} \sub{\sg_{i_3+1}} \ldots \sub{\sg_{i_k-1}}
\sub{\sg_{i_k+1}} \ldots \sub{\sg_{n-1}}. 
\]
But $w$ is obtained by letting $\sub{\sg_{i_2}} \sub{\sg_{i_3}} \ldots \sub{\sg_{i_k}}$ act on $v$ from the right, i.e.
\[
  \sub{\sg_1} \sub{\sg_2} \ldots \sub{\sg_{i_2-1}} \sub{\sg_{i_2+1}} \ldots \sub{\sg_{i_3-1}} \sub{\sg_{i_3+1}} \ldots \sub{\sg_{i_k-1}} \sub{\sg_{i_k+1}} \ldots \sub{\sg_{n-1}} \sub{\sg_{i_2}} \sub{\sg_{i_3}} \ldots \sub{\sg_{i_k}}
\]
is an $\sub{S}$-expression for $w$ without repeated letters, and thus $w$
is Boolean by Corollary \ref{BooleanInvolutionWords}. 
\end{proof}

Let $w \in \ISn$ and let $i \in [n]$ be a non-fixed point of
$w$. Then, we can delete the 2-cycle $(i,w(i))$ by multiplication of
$w$ with $(i,w(i))$ from the right. This does not change the entries
of $w$ except in the positions $i$ and $w(i)$ and we have $v = w \,
(i,w(i)) < w$ with $i$ and $w(i)$ being fixed points of $v$. 

If $w \in \ISn$ is such that $i \in [n]$ is an excedance and $j \in [n]$ is a fixed point with $i<j<w(i)$, then
we can shrink the cycle $(i,w(i))$ by conjugation with $(j,w(i))$
without changing $w$ except in the positions $i$, $j$ and $w(i)$. We
get $v= (j,w(i)) w (j,w(i)) < w$ where $(i,j)$ and $w(i)$
are a cycle respectively a fixed point of $v$.

\begin{proposition}[\bf A necessity criterion]
\label{prop:BooleanInvolutionNecessary}
Let $w \in \ISn$. If there is a long-crossing pair $(i,j)$ in $w$, then
$w$ is not Boolean. 
\end{proposition}

\begin{proof}
Fix $i,j \in [n]$ such that $(i,j)$ is a long-crossing pair in
$w$. Following our above remarks, we can delete all cycles except
$(i,w(i))$ and $(j,w(j))$ and get an involution $v \leq w$ whose only
non-fixed points are $i,j,w(i),w(j)$. Now we can shrink the remaining
two cycles so that we finally get an involution $x$ with cycles
$(j-1,j+2)$ and $(j,j+1)$ in the following way: conjugation of $v$
with $(j+1,w(j))$ yields $u \leq v$ with $u(j) = j+1$. Then we can
conjugate $u$ with $(i,j-1)$ and $(j+2,w(i))$ and get $x \leq u$
having the $2$-cycles $(j-1,j+2)$ and $(j,j+1)$ and fixed points in
all other positions. (Here, $(k,k)$ for any $k \in [n-1]$ is just the identity permutation.) A reduced $\sub{S}$-expression for $x$ is given by $\sub{\sg_{j-1}} \sub{\sg_j} \sub{\sg_{j+1}} \sub{\sg_j}$ and thus $x$ is not Boolean. But we have $x \leq u \leq v \leq w$ and therefore $w$ is not Boolean either.
\end{proof}

\begin{example}
In Figure \ref{fig:Proof_NonBoolean_Steps} the steps of the proof of Proposition \ref{prop:BooleanInvolutionNecessary} are demonstrated for $w = 5764132$ and the long-crossing pair $(1,2)$.
\end{example}

\begin{figure}[ht]
\begin{center}
\hfill
\includegraphics[width=0.4\textwidth]{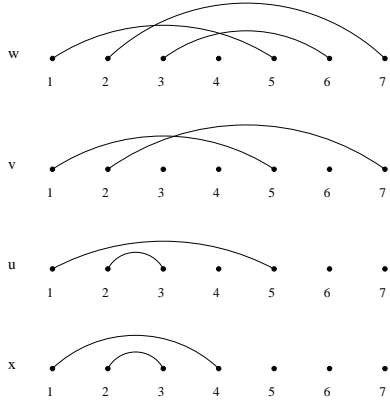}
\hfill \hfill
\caption{Illustration for the proof of Proposition \ref{prop:BooleanInvolutionNecessary}.}
\label{fig:Proof_NonBoolean_Steps}
\end{center}
\end{figure}

In fact, we have shown that $w \in \ISn$ is Boolean if and only if
$B(w)$ contains no element of the form $\sub{\sg_{j-1}} \sub{\sg_j}
\sub{\sg_{j+1}} \sub{\sg_j}$. Using similar terminology as in
\cite{Tenner2007}, such an element may be called a
{\em shift} of $\sub{\sg_1} \sub{\sg_2} \sub{\sg_3} \sub{\sg_2} =
4321 \in \I(\mathfrak{S}_4)$. Thus, $4321$ in some sense is the unique minimal
non-Boolean involution.

\begin{proposition}[\bf A pattern criterion]
\label{prop:BooleanInvolutionPatterns}
Let $w \in \ISn$. There is a long-crossing pair $(i,j)$ in $w$ if and
only if $w$ contains one or more of the patterns $4321$, $45312$ and $456123$.
\end{proposition}

\begin{proof}
"$\Rightarrow$". Let $(i,j)$ be a long-crossing pair in $w$. If $w$ contains the pattern $4321$ we are done. Thus, we can assume that $w$ avoids $4321$. In particular, this implies $w(i) < w(j)$. If $j+1$ is a fixed point then $w$ contains the pattern $45312$. Otherwise, we have $w(j+1) < i$ or $w(j+1) > w(j)$ because we assumed $w$ to be $4321$-avoiding. But then $w$ contains $456123$.

"$\Leftarrow$". We distinguish three cases. First, assume that $w$ contains $4321$ and that $\langle 4321 \rangle$ is an occurrence. Then, $\langle 3 \rangle$ or $\langle 2 \rangle$ is not a fixed point of $w$; denote that value by $k$. If $w(k)>k$, then $w(\langle 1 \rangle)>w(k)>k>\langle 1 \rangle$ and $(\langle 1 \rangle,k)$ is a long-crossing pair in $w$. Otherwise, it follows that $w(\langle 4 \rangle)<w(k)<k<\langle 4 \rangle$ and $(w(\langle 4 \rangle),w(k))$ is such a pair.

Next, assume that $w$ avoids $4321$ but contains $45312$. Let $\langle 45312 \rangle$ be an occurrence. Then $\langle 3 \rangle$ is a fixed point, because otherwise $w$ will contain $4321$ by similar arguments as in the first case. This implies that $(\langle 1 \rangle,\langle 2 \rangle)$ is a long-crossing pair.

Finally, assume that $w$ avoids $4321$ and $45312$ and let $\langle 456123 \rangle$ be an occurrence of $456123$ in $w$. The fact that $w$ avoids $45312$ implies that none of $\langle 1 \rangle,\langle 2 \rangle, \ldots \langle 6 \rangle$ is a fixed point. Furthermore, if $\langle 1 \rangle$, $\langle 2 \rangle$ or $\langle 3 \rangle$ is a deficiency, denote that value by $k$. Then $w(\langle 4 \rangle)<w(k)<k<\langle 4 \rangle$ and $w$ contains $4321$ in contradiction to our assumption. Thus, $\langle 1 \rangle$, $\langle 2 \rangle$ and $\langle 3 \rangle$ are excedances. If $w(\langle 1 \rangle) > \langle 3 \rangle$ then $(\langle 1 \rangle,\langle 2 \rangle)$ is a long-crossing pair in $w$. Otherwise, $(w(\langle 5 \rangle),\langle 3 \rangle)$ is one.
\end{proof}

\begin{figure}[ht]
\begin{center}
\hfill
\subfigure[$4321$]{
\includegraphics[width=0.15\textwidth]{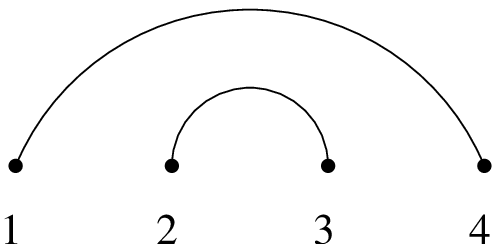}
}
\hfill
\subfigure[$45312$]{
\includegraphics[width=0.2\textwidth]{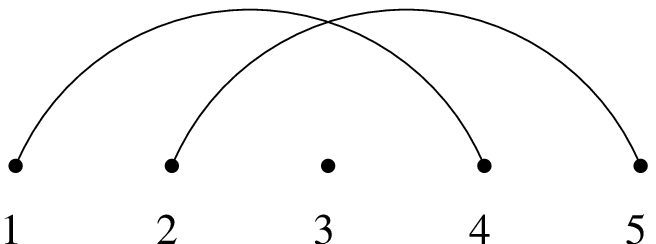}
}
\hfill
\subfigure[$456123$]{
\includegraphics[width=0.25\textwidth]{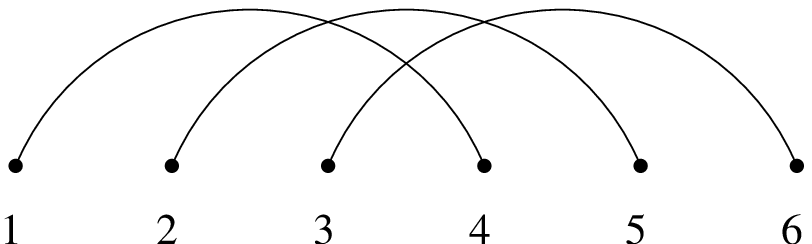}
}
\hfill
\hfill
\caption{Non-Boolean patterns for $\ISn$.}
\label{fig:Sn_Patterns}
\end{center}
\end{figure}

Let us remark that the proof of Proposition \ref{prop:BooleanInvolutionPatterns} shows that an occurrence of one of the patterns $4321$, $45312$ and $456123$ in an involution $w \in \ISn$ implies that there actually is an induced occurrence of one of those patterns.

\subsection{Other Coxeter groups}
\label{SubsectionMoreExamples}

The knowledge we gained in Section \ref{BooleanInvolutionsOfIn} about
Boolean involutions in $\ISn$ can be used to classify Boolean
involutions in $\Iti$ for some other $W$. Here, we shall develop
results for the case that $W$ is the group of signed permutations
$\SBn$. This is the group of permutations $\pi$ of the set $[\pm n] =
\{-n, \dots, -1, 1, \dots, n\}$ such that $\pi(i) = -\pi(-i)$ for all
$i\in [n]$. 

Let $\sg_i^\prime = (-i,-i-1)$, $i > 0$, and $\sg_0 = (1, -1)$. Define
$\sg_i^B = \sg_i \sg_i'$, $i > 0$, and $\sg_0^B = \sg_0$. Then,
$\SBn$ is generated as a Coxeter group by $\{\sg_0^B, \dots,
\sg_{n-1}^B\}$, whereas the symmetric group $\mathfrak{S}([\pm n])$ is generated by
$\{\sg_{n-1}^\prime, \dots, \sg_1^\prime, \sg_0, \dots, \sg_{n-1}\}$.

We have an obvious inclusion $\ISBn \subseteq \I(\mathfrak{S}([\pm n]))$; let
$\phi$ denote the inclusion map.

\begin{lemma}
Let $w \in \ISBn$. Then, $\phi(w \sub{\sg^B_0}) = \phi(w) \sub{\sg_0}$. Furthermore, for $i \in [n-1]$,
\[
\phi(w \sub{\sg^B_i}) =
\begin{cases}
\phi(w) \sub{\sg_i} & \text{if } \sg_i w \sg_i = \sg^\prime_i w \sg^\prime_i \not= w, \\  
\phi(w) \sub{\sg_i} \sub{\sg^\prime_i} & \text{otherwise}. 
\end{cases} 
\]
\end{lemma}

\begin{proof}
Let $w\in \ISBn$. Assume first that $w = \sg_i w \sg_i$. This implies $w = \sg_i' w \sg_i'$ as well as $w \sg^B_i = \sg^B_i w$ and thus
$ \phi(w \sub{\sg_i^B}) = \phi(w \sg_i^B) = \phi(w) \sg_i \sg_i' =
\phi(w) \sub{\sg_i} \sub{\sg_i'}$.
If, on the other hand, $\sg_i w \sg_i \not= \sg_i' w \sg_i'$ it
follows that $\sg_i w \sg_i \neq w \neq \sg_i' w \sg_i'$ and thus
$
\phi(w \sub{\sg_i^B}) = \phi(\sg_i^B w \sg_i^B) = \sg_i \sg_i' \phi(w) \sg_i \sg_i' = \phi(w) \sub{\sg_i} \sub{\sg_i'}
$.
Finally, assume that $w \not= \sg_i w \sg_i = \sg_i' w \sg_i'$. By the
remark after Proposition \ref{pr:orbit}, $\ell(\sg_i \phi(w) \sg_i) =
\ell(\phi(w)) \pm 2$. Assume the 
plus sign holds; otherwise a completely analogous argument applies. We
claim that $\phi(w) \sg_i = \sg_i' \phi(w)$ and $\sg_i \phi(w) =
\phi(w) \sg_i'$. To see this, consider the open interval $I = (\phi(w),\sg_i
\phi(w) \sg_i)$ in the Bruhat order on $\mathfrak{S}([\pm n])$. Known
facts about the Bruhat order (see e.g.\ \cite[Lemma 2.7.3]{Bjoerner2005}) imply
that $I$ consists of exactly two elements. Thus, $I = \{\phi(w)
\sg_i', \sg_i' \phi(w)\} = \{ \phi(w)
\sg_i, \sg_i \phi(w)\}$, proving the claim. We conclude $\phi(w
\sub{\sg_i^B}) = \phi(w \sg_i^B) = \phi(w) \sg_i \sg_i' = \sg_i
\phi(w) \sg_i = \phi(w) \sub{\sg_i}$.
\end{proof}

In conjunction with Proposition \ref{BooleanInvolutionWords}, this in
particular implies

\begin{corollary}
\label{cor:BooleanInclusionMap}
An involution $w \in \ISBn$ is Boolean if and only if $\phi(w) \in
\I(\mathfrak{S}([\pm n]))$ is Boolean. 
\qed
\end{corollary}

There are several possible ways to extend the notion of pattern avoidance from
$\Sn$ to $\SBn$. We now describe the version which we shall use. 

Given $\pi \in \SBn$ and $p\in \SBm$ (with $m\le n$), we say that
$\pi$ {\em contains the signed pattern} $p$ if there exist $1\leq
i_1<\cdots <i_m \leq n$ such that $(|\pi(i_1)|, \dots, |\pi(i_m)|)$ is
an occurrence of the (unsigned) pattern $|p(1)|\cdots|p(m)|$ in the
ordinary sense, and $\sgn(\pi(i_j))=\sgn(p(j))$ for all $j\in [m]$.

We have a characterization of the Boolean elements of
$\I(\mathfrak{S}([\pm n]))$ in terms of patterns. This can be
translated into signed pattern avoidance in $\ISBn$.

Below, we use window notation for signed permutations. Thus, $\pi
\in \SBn$ is represented by the sequence $\pi(1)\cdots \pi(n)$. For
compactness, we write $\sub{i}$ instead of $-i$. As an example, $\sub{2}31$
denotes the signed permutation defined by $\pm 1 \mapsto \mp 2$, $\pm
2 \mapsto \pm 3$ and $\pm 3 \mapsto \pm 1$.

\begin{proposition}
\label{SBn_Pattern_List}
Let $w \in \ISBn$. Then $w$ is Boolean if and only if it avoids all of
the following signed patterns. 
\begin{center}
$\begin{array}{lll}
        4321  			& 45312 			& 456123 \\
	\sub{1} \sub{2} 	& 1 \sub{3} \sub{2} 		& \sub{3} \sub{2} \sub{1} \\
	2 1 \sub{3} 		& 4 2 \sub{3} 1 		& 4 \sub{3} \sub{2} 1 \\
	3 \sub{4} 1 \sub{2}     & \sub{4} 5 3 \sub{1} 2         & 4 5 \sub{3} 1 2 \\
	\sub{4} 3 2 \sub{1}     & 5 \sub{4} 3 \sub{2} 1         & \sub{4} 5 6 \sub{1} 2 3 \\
& & 	5 \sub{4} 6 \sub{2} 1 3	
\end{array}$
\end{center}
\end{proposition}

\begin{proof}

Recalling Theorem \ref{th:BooleanInvolution} and Corollary
\ref{cor:BooleanInclusionMap}, we need only show that $w$ contains
one of the patterns from the above list if and only if $\phi(w)$
contains $4321$, $45312$ or $456123$. 

"$\Rightarrow$". It is straightforward to check, that if $w$ contains
any of the signed patterns listed in the lemma, then $\phi(w)$
contains $4321$, $45312$ or $456123$. For example, assume that $w$ contains $2 1 \sub{3}$. Then the definition of $\phi$ implies that $\phi(w)$ contains $3 \sub{1} \sub{2} 2 1 \sub{3}$, which in turn contains $4321$.

"$\Leftarrow$". Recall from the proof of Proposition
\ref{prop:BooleanInvolutionPatterns} that $\phi(w)$ contains $4321$,
$45312$ or $456123$ if and only if it has an induced occurrence of one
of those three patterns. We will show for $4321$ that such an induced
occurrence in $\phi(w)$ implies that $w$ contains one of the signed
patterns listed in the lemma. Similar arguments apply in the other
cases. 

Assume that $\phi(w)$ contains an induced $4321$-pattern. The graph
representation of $\phi(w)$ is symmetric with respect to the vertical
axis bisecting the segment between $1$ and $-1$, because $\phi(w)$ is
the image of a signed permutation. In Figure
\ref{fig:SBn_Signed_Patterns} we have indicated with thick edges all
possibilities of how the occurrence of $4321$ can be placed in the
graph representation and completed to a symmetric pattern. This leads
to the list of signed patterns in the first column of the
proposition.
\end{proof}

\begin{figure}[ht] 
  \centering
  \hfill
  \subfigure[$w$ contains $4321$]{
    \includegraphics[width=.4\textwidth]{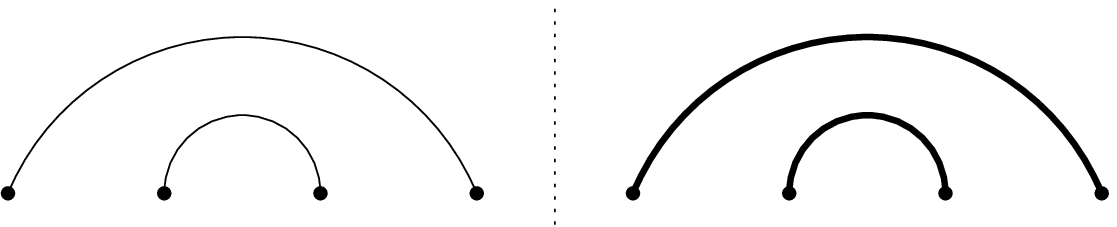} 
  }
  \hfill
  \subfigure[$w$ contains $\sub{4} 32 \sub{1}$]{
    \includegraphics[width=.4\textwidth]{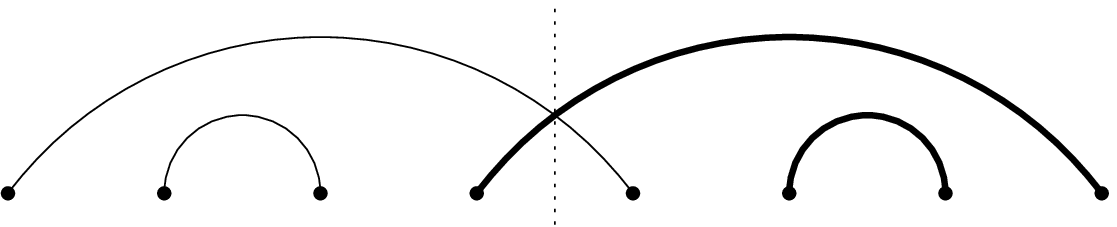} 
  }
  \hfill \\ \hfill  
  \subfigure[$w$ contains $3 \sub{4} 1 \sub{2}$]{
    \includegraphics[width=.4\textwidth]{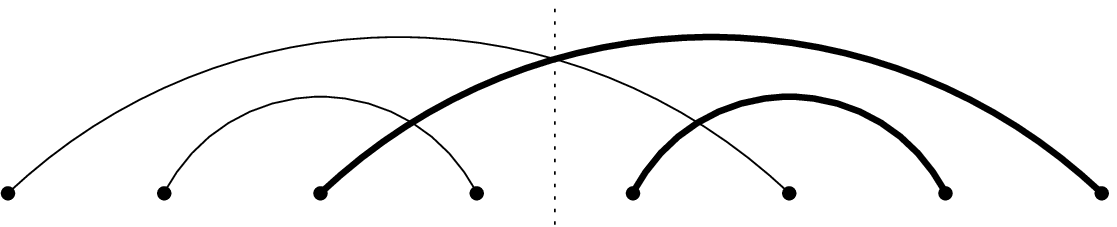} 
  }
  \hfill
  \subfigure[$w$ contains $21 \sub{3}$]{
    \includegraphics[width=.4\textwidth]{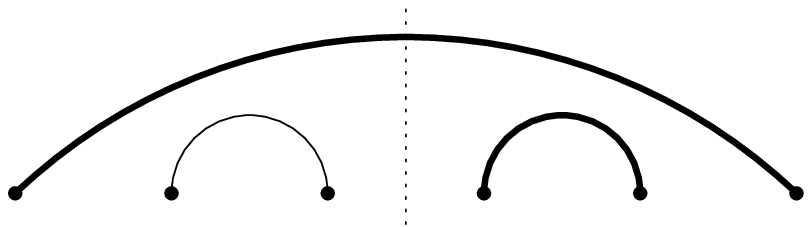} 
  }
  \hfill \\ \hfill
  \subfigure[$w$ contains $\sub{1} \sub{2}$]{
    \includegraphics[width=.4\textwidth]{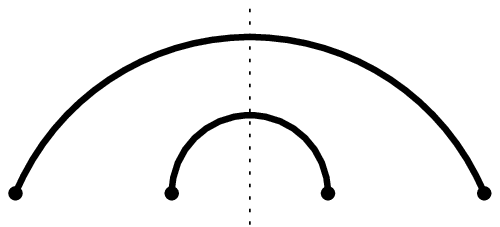} 
  }
  \hfill \hfill
  \caption{Graph representations for $4321$-containing $\phi(w)$.}
  \label{fig:SBn_Signed_Patterns}
\end{figure}

The subgroup $\SDn$ of $\SBn$ consists of the permutations with an even
number of negative elements in the window notation. It is a Coxeter
group in its own right. The interested reader is referred to
\cite[Corollary 5.25]{Vorwerk2007} for a list
of forbidden patterns that characterize Boolean involutions in
$\SDn$. However, the obvious analogue of Corollary
\ref{cor:BooleanInclusionMap} does not hold. This makes the
proofs more technical. 

\section{Enumeration}\label{SectionEnumeration}
In this section we shall deduce some enumerative facts about Boolean
involutions. The key is a simple linear recurrence formula valid for a
class of Coxeter groups which we now specify.

Let $W$ be a Coxeter group with Coxeter generator set $S = \{s_1,
\dots, s_n\}$, $n \ge 3$, such that $s_n$ commutes with all $s_i$ for $i \le
n-2$. Further, assume $s_{n-1}$ commutes with all $s_i$ for $i \le
n-3$. Finally, suppose $s_ns_{n-1}s_n = s_{n-1}s_ns_{n-1}$ and
$s_{n-1}s_{n-2}s_{n-1} = s_{n-2}s_{n-1}s_{n-2}$. This means that the
Coxeter graph of $(W,S)$ is of the form displayed in Figure
\ref{fig:CoxeterGraph_Recursion}. Examples of such $W$ include $\Sn$,
$n\ge 3$, as well as $\SBn$, $n\ge 4$, and $\SDn$, $n\ge 5$.

\begin{figure}[ht]
  \centering
  \includegraphics[width=.45\textwidth]{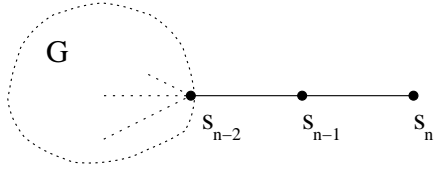}
  \caption{The Coxeter graph of $W = W_n$.}
  \label{fig:CoxeterGraph_Recursion}
\end{figure}

For brevity, denote by $W_i$, $i\in [n]$, the standard parabolic subgroup of
$W$ generated by $\{s_1, \dots, s_i\}$. Let $f(W_i,l,a)$ be the number of
Boolean involutions in $W_i$ of Coxeter length $l$ and absolute length
$a$. In other words,
\[
f(W_i,l,a) = |\{w\in \I(W_i) : w\text{ is Boolean, }\ell(w)=l\text{ and }
\al(w) = a\}|. 
\]

\begin{theorem}
Let $(W,S)$ be as above. Then, for $n,l\ge 3$ and $a\ge 1$,
\begin{equation}
\label{equ:BooleanInvolutionRecursion}
\begin{split}
f(W_n,l,a) = \,\, & f(W_{n-1},l,a) + f(W_{n-1},l-2,a)\\
 & + f(W_{n-2},l-1,a-1) - f(W_{n-2},l-2,a) \\
 & + f(W_{n-2},l-3,a-1) - f(W_{n-3},l-3,a-1).\\
\end{split}
\end{equation}
\begin{proof}
Suppose $w \in \I(W_n)$ is Boolean with $\ell(w) = l$ and
$\al(w)=a$. If $s_n \not\leq w$ then $w$ is a Boolean involution in
$W_{n-1}$. There are exactly $f(W_{n-1},l,a)$ such $w$. Otherwise,
consider the lexicographically first (with respect to the indices of
the generators) reduced $\sub{S}$-expression for
$w$; call this expression $E$. We have two cases, depending on whether
$E$ ends with $\sub{s_n}$. If it does not, then it necessarily ends
with $\sub{s_n}\sub{s_{n-1}}$.

{\bf Case 1, $E$ ends with $\sub{s_n}$.} This is the case if and only
if $w \sub{s_n} \in \I(W_{n-1})$. If $w$ commutes with $s_n$, we
have $\ell(w) = 
\ell(w\sub{s_n}) + 1$ and $\al(w) = \al(w\sub{s_n})+1$. If not, $\ell(w) =
\ell(w\sub{s_n}) + 2$ and $\al(w) = \al(w\sub{s_n})$. 

Now, $w$ commutes with $s_n$ if and only if $\sub{s_{n-1}}$ does not
occur in $E$, i.e.\ if and only if $w\sub{s_n}\in
\I(W_{n-2})$. Hence, the number of $w$ that fall into Case 1 is
$f(W_{n-2},l-1,a-1) + f(W_{n-1},l-2,a) - f(W_{n-2},l-2,a)$.

{\bf Case 2, $E$ ends with $\sub{s_n}\sub{s_{n-1}}$.} Let $u =
w\sub{s_{n-1}}\sub{s_n}$. We are in Case 2 if and only if
$u \in \I(W_{n-2})\setminus
\I(W_{n-3})$. Then, $u$ commutes with $s_n$ whereas 
$u\sub{s_n}$ does not commute with $s_{n-1}$. Hence, $\ell(w) =
\ell(u)+3$ and $\al(w) = \al(u)+1$. Consequently, there are
$f(W_{n-2},l-3,a-1)-f(W_{n-3},l-3,a-1)$ elements $w$ that belong to Case 2.  
\end{proof}
\end{theorem}

\begin{corollary}
Keeping the above assumptions on $(W,S)$, let $g(W_i,k)$ denote the
number of Boolean involutions $w \in 
\I(W_i)$ with rank $\rho(w) = k$. Also, define $h(W_i)$ to be
the number of Boolean involutions in $\I(W_i)$. Then, for $n\ge 3$
and $k \ge 2$,
\[
g(W_n,k) = g(W_{n-1},k) + g(W_{n-1},k-1) + g(W_{n-2},k-2) - g(W_{n-3},k-2)
\]
and
\[
h(W_n) = 2h(W_{n-1})+h(W_{n-2})-h(W_{n-3}).
\]
\begin{proof}
Once we recall that $\rho(w) = (\ell(w)+\al(w))/2$, the identities
follow by summing equation \ref{equ:BooleanInvolutionRecursion} over
appropriate $l$ and $a$.
\end{proof}
\end{corollary}

From now on, let us stick to the case of symmetric groups. With $W =
\mathfrak{S}_{n+1}$, we have $W_j = \mathfrak{S}_{j+1}$ and
$f(\mathfrak{S}_j,i,e)$ is
the number of Boolean involutions in $\mathfrak{S}_j$ with $i$
inversions and $e$ excedances.

\begin{proposition}
Consider the generating function for the number of Boolean
involutions in $\Sn$ with respect to inversion number and excedance
number. That is, define
\[
F(x,y,z) = \sum_{n\ge 1,\, i\ge 0,\, e\ge 0}f(\Sn,i,e)x^ny^iz^e.
\]
Then,
\[
F(x,y,z) =
\frac{x^2yz+x-x^2y^2-x^3y^3z}{1-x-x^2yz-xy^2+x^2y^2-x^2y^3z+x^3y^3z}. 
\]
\begin{proof}
This follows from equation \ref{equ:BooleanInvolutionRecursion} via
standard techniques once one has computed $f(\Sn,l,a)$ for $n\leq 3$
or $i\leq 2$ or $e=0$. These numbers vanish except in the following
cases: $f(\Sn,0,0)=1$ ($n\ge 1$), $f(\Sn,1,1)=n-1$ ($n\ge 2$),
$f(\Sn,2,2)=(n^2-5n+6)/2$ ($n\ge 4$) and $f(\mathfrak{S}_3,3,1)=1$.
\end{proof}
\end{proposition}

Plugging in $y = z = t^{1/2}$ and $y = z = 1$, one
obtains the generating functions for $g(\Sn,k)$ and $h(\Sn)$, respectively.
\begin{corollary}
We have
\[
\sum_{n\ge 1,\, k\ge 0}g(\Sn,k)x^nt^k = \frac{x(1-x^2t^2)}{(1-x^2t^2)(1-x)-xt}
\]
and
\[
\sum_{n\ge 1}h(\Sn)x^n = \frac{x(1-x^2)}{1-2x-x^2+x^3}.
\]
\qed
\end{corollary}

Recall that a {\em Motzkin path of length $n$} is a lattice path from
$(0,0)$ to $(n,0)$ which never goes below the $x$-axis and whose steps
are either $(1,1)$, $(1,0)$ or $(1,-1)$. These steps are called {\em
upsteps}, {\em flatsteps} and {\em downsteps}, respectively. We
denote by $M_n$ the set of Motzkin paths of length $n$. 

The sequence $h(\Sn)$, $n \geq 1$, can be found in
\cite[A052534]{Sloane} where it is 
referred to as the number of Motzkin paths with certain
properties. Let $M_n^\rest \subseteq M_n$ denote the set of 
Motzkin paths of length $n$ that never go higher than level $2$ and
whose flatsteps all occur on level at most $1$. We call a path
in $M_n^\rest$ a {\em restricted Motzkin path} of length $n$.

\begin{proposition} \label{pr:bijection}
Let $\psi: \ISn \rightarrow M_n$ be the mapping which sends
an involution $w$ to the Motzkin path $\psi(w)$ with a flatstep,
upstep or downstep as $k$-th step if $w(k)$ is a fixed point, an
excedance or a deficiency, respectively. Then $\psi$ induces a
bijection between the Boolean involutions in $\ISn$ and the restricted
Motzkin paths of length $n$. 
\end{proposition}

An example is shown in Figure \ref{fi:Motzkin}.

\begin{proof}[Proof of Proposition \ref{pr:bijection}]
For every $w \in \ISn$, $\psi(w)$ is a lattice path by definition. It
goes from $(0,0)$ to $(n,0)$, because $w$ has the same number of
excedances and deficiencies, and it obviously does not go below
the $x$-axis. Thus, $\psi(w)$ is a Motzkin path for all $w \in \ISn$
and $\psi$ is well-defined.

Assume that the $k$-th step of $\psi(w)$ is a flatstep on level $p$
(i.e. it goes from $(k-1,p)$ to $(k,p)$). Then there are exactly $p$
elements $l \in [k-1]$ such that $w(l) > k$. If $p>1$ there are
$l_1,l_2 \in [k-1]$ such that $w(l_1) > k$ and $w(l_2) > k$. Assuming
$l_1 < l_2$, $(l_1,l_2)$ is a long-crossing pair. Thus, if $\psi(w)$ is a path
with a flatstep on level 2 or higher, then $w$ is not
Boolean. Similarly, it follows that if $\psi(w)$ goes to a level $>2$,
then $w$ is not Boolean. Therefore every Boolean involution is mapped
to a restricted Motzkin path and 
\[
\psi(\{w \in \ISn: w \text{ is Boolean}\}) \subseteq M_n^\rest.
\]

In order to show the reverse inclusion, fix a restricted Motzkin path
$\Gamma$. We 
construct an involution $w \in \ISn$ such that $\psi(w)=\Gamma$. For $k
\in [n]$ define $w(k) = k$ if the $k$-th step of $\Gamma$ is a
flatstep. If the $k$-th step is an upstep or a downstep, and it
is the $m$-th upstep or downstep, respectively, then
define $w(k) = p$ where $p$ is such that the $p$-th step in $\Gamma$
is the $m$-th downstep or upstep, respectively. This obviously
defines a unique involution in $\ISn$. Observe that the given
restrictions on the Motzkin path ensure that long-crossing pairs never
occur. Hence, the constructed
involution is Boolean. This proves $\psi(\{w \in \ISn: w \text{ is
Boolean}\}) = M_n^\rest$.

Note that the proof of Proposition
\ref{prop:BooleanInvolutionSufficient} implies that a Boolean
involution is uniquely determined by its sets of excedances and
deficiencies. Thus, $\psi$ yields a bijection between the Boolean
elements of $\ISn$ and $M_n^\rest$. 
\end{proof}

\begin{figure}[htb] 
\begin{center}
\includegraphics[width=.6\textwidth]{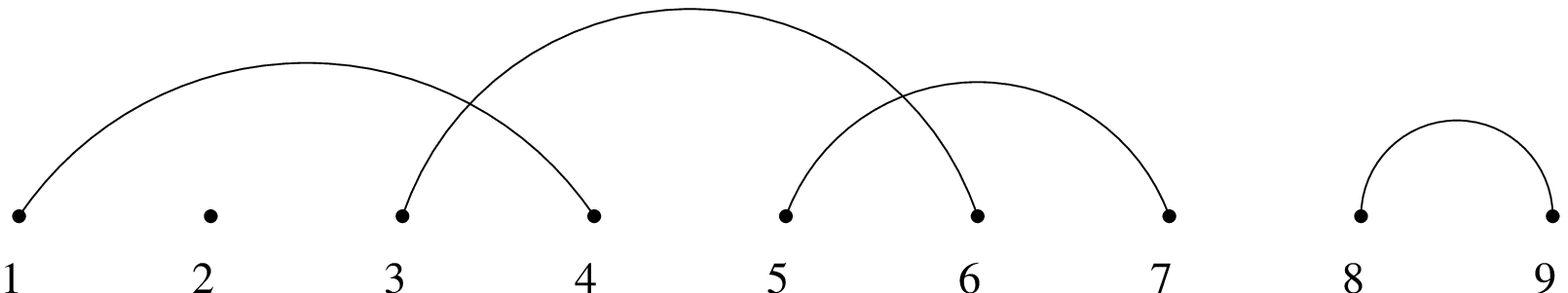} 

\vspace{12pt}

\includegraphics[width=.6\textwidth]{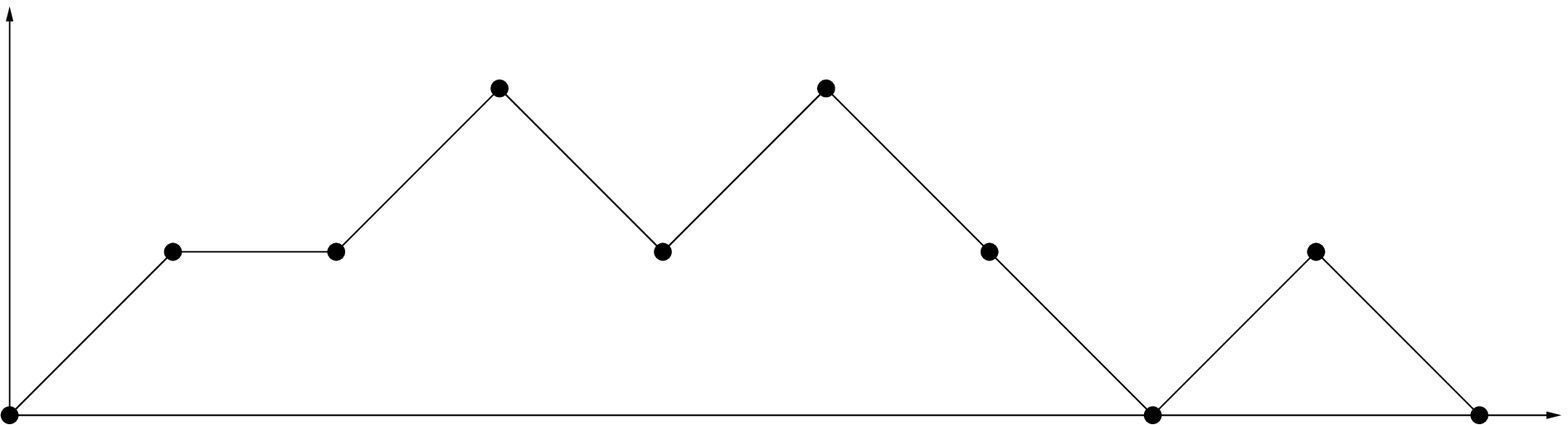}
\end{center}
\caption{A Boolean involution and the corresponding restricted Motzkin
path.} \label{fi:Motzkin}
 \end{figure}

We conclude this section by pointing out what happens to our favourite
statistics under the bijection $\psi$.

\begin{proposition} \label{pr:rank}
Suppose $w\in \ISn$ is Boolean. Let $\alpha(w)$ be the number of
indices $i\in [n]$ such that $\psi(w)$ contains the point
$(i,0)$. Then, $\rho(w) = n-\alpha(w)$.
\begin{proof}
Because $w$ is Boolean, $\rho(w)$ is the number of distinct generators
$\sg_i$, $i\in [n-1]$ that appear in reduced $\sub{S}$-expressions 
for $w$, i.e.\ that are below $w$ in the Bruhat order. On
the other hand, for $i\in [n-1]$, $(i,0)$ belongs to $\psi(w)$ if and
only if $w(j)\le i$ for all $j \le i$. This holds if and only if
$\sg_i \not \le w$.
\end{proof}
\end{proposition}

By construction, the number of excedances (or deficiencies) of $w$ is
precisely the number of upsteps (or downsteps) of $\psi(w)$. Since
$2\rho = \exc + \inv$, Proposition \ref{pr:rank} also provides an
interpretation for the inversion number of $w$ in terms of the
corresponding Motzkin path.

As an example, the path in Figure \ref{fi:Motzkin} touches the
$x$-axis in two points (excluding the origin). Thus, the rank of the
corresponding involution $w$ is $9-2=7$. There are four upsteps, so
$\exc(w)=4$. Hence, $\inv(w)=10$. 

\section{Twisted involutions} \label{SectionOpenProblems} 
As was mentioned in the introduction, a good reason to study $\BrI$ is
the connection with orbit decompositions of symmetric varieties which is
explained in \cite{Richardson1990}. In this
context, the more general setting of {\em twisted involutions} with
respect to an involutive automorphism $\theta$ of $(W,S)$ is
important. These are the elements $w\in W$ such that
$\theta(w)=w^{-1}$. Thus, $\Iti$ corresponds to the $\theta = \id$
case. 

In the context of a symmetric group, there is only one non-trivial
$\theta$; it is given by $w\mapsto w_0ww_0$, where $w_0\in \Sn$ is the
longest element (the reverse permutation). 

\begin{problem}\label{problem:theta}
Find an analogue of Theorem \ref{th:BooleanInvolution} valid for
$\theta \neq \id$. 
\end{problem}

In order to attack this problem, \cite[Proposition 5.1]{Vorwerk2007}
is likely to be useful. It provides a generalization to arbitrary $\theta$
of Proposition \ref{BooleanInvolutionWords}. Also, the tools mentioned
in Subsection \ref{ss:involutions} have direct counterparts in this
more general setting; see \cite{Hultman2007, Hultman2007a}.

We remark that whenever $\theta$ is given by $w\mapsto w_0ww_0$, the
Bruhat order on twisted involutions is isomorphic to the dual of
$\BrI$. Thus, Problem \ref{problem:theta} is equivalent to the
problem of characterizing Boolean principal order filters in
$\mathrm{Br}(\I(\Sn))$.

\end{document}